\documentclass[final,3p,times]{elsarticle}       

\usepackage{lineno} 
\modulolinenumbers[5]

\makeatletter
\def\ps@pprintTitle{%
  \let\@oddhead\@empty
  \let\@evenhead\@empty
  \let\@oddfoot\@empty
  \let\@evenfoot\@oddfoot
}
\makeatother
\usepackage[table,xcdraw,svgnames]{xcolor}
\usepackage[colorlinks]{hyperref}
\AtBeginDocument{%
  \hypersetup{
    citecolor=magenta,
    linkcolor=blue,
    urlcolor=Blue}}
\usepackage[heading=true,scheme=plain]{ctex} 
\usepackage{mathrsfs}
\usepackage{amsfonts}   
\usepackage{amssymb}   
\usepackage{lmodern}   
\usepackage{bm}       
\usepackage{bbm}  

\usepackage{autobreak}       
\allowdisplaybreaks[4]   

 \usepackage[capitalise]{cleveref}

  \usepackage{extarrows}  
\usepackage{tikz}   
\usetikzlibrary{arrows.meta}
\usetikzlibrary{snakes}
\usepackage{pgfplots}
\usetikzlibrary{patterns}
\usepackage{tkz-euclide}
\usetikzlibrary{calc}
\usetikzlibrary{intersections}
\usepackage{subfigure}       
\numberwithin{figure}{section}
\usepackage[justification=centering]{caption}  
 \usepackage[shortlabels]{enumitem}  
\numberwithin{equation}{section} 
 \allowdisplaybreaks[4]   
\usepackage{amsthm}
\newtheorem{definition}{Definition} [section]            
\newtheorem{theorem}{Theorem}[section]            
\newtheorem{lemma}{Lemma} [section]
\newtheorem{proposition}{Proposition}[section]
\newdefinition{corollary}{Corollary}[section]
\newtheorem{remark}{Remark}

\def\dchi{\scalebox{1.2}{$\chi$}}

\def\dint{\displaystyle\int}

\def\trace{{\rm{trace}}~}          

\DeclareMathOperator*{\bmo}{BMO}

\DeclareMathOperator*{\lip}{Lip}

\newcommand{\mathd}{\mathrm{d}}

\biboptions{numbers,sort&compress}   

\begin{document}

\begin{frontmatter}

\title{{\bfseries  Some estimates for commutators of the fractional  maximal function on stratified Lie groups }}

\author[mymainaddress]{Jianglong Wu\corref{mycorrespondingauthor}} %
\cortext[mycorrespondingauthor]{Corresponding author}
\ead{jl-wu@163.com}

\author[mysecondaryaddress]{Wenjiao Zhao}

\address[mymainaddress]{Department of Mathematics, Mudanjiang Normal University, Mudanjiang  157011, China}
\address[mysecondaryaddress]{School of Mathematics, Physics and Finance, Anhui Polytechnic University, Wuhu 241000, China}

\begin{abstract}
In this paper, the main aim  is to consider the boundedness of  the   fractional maximal commutator $M_{\alpha,b}$ and the nonlinear commutator   $[b, M_{\alpha}]$  on the Lebesgue spaces over some stratified Lie group  $\mathbb{G}$ when $b$ belongs to the Lipschitz space, by which some new characterizations of the Lipschitz spaces on Lie group  are given.
\end{abstract}


\end{frontmatter}


\section{Introduction and main results}
\label{sec:introduction}

During the last several decades, stratified groups appear in quantum physics and many parts of mathematics, including  several complex variables, Fourier analysis, geometry, and topology \cite{folland1982hardy,varopoulos2008analysis}.
The geometry structure of stratified groups is so good that it inherits a lot of analysis properties from the Euclidean spaces \cite{stein1993harmonic,grafakos2009modern}.
Apart from this, the difference between the geometry structures of Euclidean spaces and stratified groups makes the study of function spaces on them more complicated. However, many harmonic analysis problems  on stratified Lie groups deserve a further investigation since  most results   of the theory of Fourier transforms and distributions in Euclidean spaces  cannot yet be duplicated.

It is worthwhile to note that the  fractional maximal operator plays an important role in real and harmonic analysis and applications, such as potential theory and partial differential equations (PDEs),  since it  is intimately related to the Riesz potential operator, which  is a powerful tool in the study of  the smooth function spaces  (see \cite{folland1982hardy,bonfiglioli2007stratified,carneiro2017derivative}). On the other hand,  there are two major reasons  why the study of the  commutators  has got widespread attention. The first one is that the boundedness of commutators can produce some characterizations of function spaces \cite{janson1978mean,paluszynski1995characterization}. The other one is that the theory of commutators is intimately related to the regularity properties of the solutions of certain PDEs  \cite{chiarenza1993w2,difazio1993interior,ragusa2004cauchy,bramanti1995commutators}.

Let $T$ be the classical singular integral operator. The Coifman-Rochberg-Weiss type commutator $[b, T]$ generated by $T$ and a
suitable function $b$ is defined by
\begin{align} \label{equ:commutator-1}
 [b,T]f      & = bT(f)-T(bf).
\end{align}
A well-known result shows that  $[b,T]$ is bounded on $L^{p}(\mathbb{R}^{n})$  for $1<p<\infty$ if and only if $b\in \bmo(\mathbb{R}^{n})$ (the space of bounded mean oscillation functions). The sufficiency was provided  by Coifman et al.\cite{coifman1976factorization} and the necessity was obtained by Janson \cite{janson1978mean}.
Furthermore, Janson \cite{janson1978mean} also established some characterizations of the Lipschitz space $\Lambda_{\beta}(\mathbb{R}^{n})$ via commutator \labelcref{equ:commutator-1} and  proved that  $[b,T]$ is bounded from $L^{p}(\mathbb{R}^{n})$ to $L^{q}(\mathbb{R}^{n})$   for $1<p<n/\beta$ and $1/p-1/q=\beta/n$ with $0<\beta<1$ if and only if  $b\in \Lambda_{\beta}(\mathbb{R}^{n})$  (see also Paluszy\'{n}ski \cite{paluszynski1995characterization}).



Denote by $\mathbb{G}$  and $\mathbb{R}$ the sets of   groups and   real numbers separately. Let $0\le \alpha <Q$ and $f: \mathbb{G} \to \mathbb{R}$ be a locally integrable function,  the fractional maximal function is given by
\begin{align*}
   M_{\alpha}(f)(x)    &=  \sup_{B\ni x  \atop B\subset \mathbb{G}}  \dfrac{1}{|B|^{1-\alpha/Q}} \dint_{B}  |f(y)| \mathd y,
\end{align*}
where the supremum is taken over all $ \mathbb{G}$-balls $B\subset \mathbb{G}$  containing $x$ with radius $r>0$ , and $|B|$   represents  the Haar measure of the $ \mathbb{G}$-ball $B$ (for the notations and notions, see  \cref{sec:preliminary} below).
When $\alpha=0$, we simply write  $M $ instead of $M_{0}$, which is the  Hardy-Littlewood maximal function defined as
\begin{align*}   
   M(f)(x)    &=  \sup_{B\ni x\atop B\subset \mathbb{G}}  \dfrac{1}{|B|} \dint_{B}  |f(y)| \mathd y.
\end{align*}


Similar to  \labelcref{equ:commutator-1},  we can define two different kinds of commutator of the fractional maximal function as follows.
\begin{definition} \label{def.commutator-frac-max}
 Let  $0 \le \alpha<n$ and $b$ be  a locally integrable function on $\mathbb{G}$.
\begin{enumerate}[label=(\roman*)]  
\item The maximal commutator of $M_{\alpha}$ with $b$ is given by
\begin{align*}
 M_{\alpha,b} (f)(x) &= \sup_{B\ni x \atop B\subset \mathbb{G}}  \dfrac{1}{|B|^{1-\alpha/Q}} \dint_{B} |b(x)-b(y)| |f(y)| \mathd y,
\end{align*}
where the supremum is taken over all $ \mathbb{G}$-balls $B\subset\mathbb{G}$ containing $x$.
 \item  The nonlinear commutators generated by  $M_{\alpha}$ and $b$  is defined by
\begin{align*}
[b,M_{\alpha}] (f)(x) &= b(x) M_{\alpha} (f)(x) -M_{\alpha}(bf)(x).
\end{align*}
\end{enumerate}
\end{definition}

When $\alpha=0$, we simply denote by $[b,M]=[b,M_{0}]$ and $M_{b} =M_{0,b} $.

We call $[b,M_{\alpha}] $ the nonlinear commutator because it is not even a sublinear operator,
although the commutator $[b,T]$ is a linear one. It is worth noting that the nonlinear commutator $[b,M_{\alpha}] $ and the maximal commutator $M_{\alpha,b}$ essentially differ from each other.
For example, $M_{\alpha,b}$ is positive and sublinear, but $[b,M_{\alpha}] $ is neither positive nor sublinear.

 In  1990, by using real interpolation techniques, Milman and Schonbek\cite{milman1990second} established a commutator result that applies to the Hardy-Littlewood maximal function as well as to a large class of nonlinear operators.
 In 2000, Bastero et al.\cite{bastero2000commutators}    proved the necessary and sufficient condition for the boundedness of  the  nonlinear    commutators $[b,M]$ and $[b,M^{\sharp}]$ on $L^{p}$ spaces. In 2009,  Zhang and Wu\cite{zhang2009commutators} studied the same problem for $[b,M_{\alpha}]$.
In 2017, Zhang\cite{zhang2017characterization} considered some new characterizations of the Lipschitz spaces via the boundedness of maximal commutator $M_{b}$ and the (nonlinear) commutator   $[b, M]$ in Lebesgue spaces and Morrey spaces on Euclidean spaces.
In 2018, Zhang et al.\cite{zhang2018commutators} gave necessary and sufficient conditions for the boundedness of the nonlinear commutators $[b,M_{\alpha}]$ and $[b,M^{\sharp}]$  on Orlicz spaces when the symbol $b$ belongs to Lipschitz spaces, and obtained some new characterizations of non-negative Lipschitz functions.
Recently, Guliyev\cite{guliyev2022some}
extended the mentioned results to  Orlicz spaces $L^{\Phi} (\mathbb{G})$ over some stratified Lie group when the symbols belong to $\bmo(\mathbb{G})$.
And Liu et al.\cite{liu2022characterisation} established the characterization of BMO spaces by the
boundedness of some commutator   in variable Lebesgue spaces. Meanwhile, Wu and Zhao \cite{wu2023characterizationlip} extended some results of \cite{zhang2017characterization} to  stratified Lie group when the symbols belong to the Lipschitz spaces.

Motivated by the papers mentioned above, the purpose of this paper is to   study the boundedness of the  fractional maximal commutator $M_{\alpha,b}$  and the nonlinear commutator   $[b, M_{\alpha}]$ on the Lebesgue spaces in the context of some stratified Lie group  $\mathbb{G}$ when $b\in \Lambda_{\beta}(\mathbb{G})$, by which some new characterizations of the Lipschitz spaces are given.

To state the results, we also give the following notations.

Let $\alpha\ge 0$, for a fixed $\mathbb{G}$-ball $B^{*}$, the fractional maximal function with respect to $B^{*}$ of a locally integrable function $f$ is given by
\begin{align*}
M_{\alpha,B^{*}} (f)(x) &= \sup_{ B\ni x \atop  B\subset B^{*}} \dfrac{1}{|B|^{1-\alpha/Q}} \dint_{B} |f(y)| \mathd y,
\end{align*}
where the supremum is taken over all $\mathbb{G}$-balls $B$  such that $x\in B\subset B^{*}$. When $\alpha= 0$,  we simply write $M_{B^{*}}$ instead of $M_{0,B^{*}}$.

Our main results can be stated as follows.

\begin{theorem} \label{thm:nonlinear-frac-max-lip}
Let  $0 <\beta <1$, $0 <\alpha <Q$, $0 <\alpha+\beta <Q$  and let $b$ be a locally integrable function on $\mathbb{G}$.    Then the following  assertions  are equivalent:
\begin{enumerate} [label=(A.\arabic*) ]
  \item   $b\in \Lambda_{\beta}(\mathbb{G})$ and $b\ge 0$.
    \label{enumerate:nonlinear-frac-max-lip-1}
   \item $[b,M_{\alpha} ]$ is bounded from $L^{p}(\mathbb{G})$ to $L^{q}(\mathbb{G})$ for all $p$ and $q$ satisfy $1 <p < \frac{Q}{\alpha+\beta}$ and $\frac{1}{q} =\frac{1}{p} - \frac{\alpha+\beta}{Q}$.
    \label{enumerate:nonlinear-frac-max-lip-2}
 \item $[b,M_{\alpha} ]$ is bounded from $L^{p}(\mathbb{G})$ to $L^{q}(\mathbb{G})$ for some  $p$ and $q$ such that  $1 <p < \frac{Q}{\alpha+\beta}$ and $\frac{1}{q} =\frac{1}{p} - \frac{\alpha+\beta}{Q}$.
    \label{enumerate:nonlinear-frac-max-lip-3}
   \item  There exists $s\in   [1,\infty)$ such that
\begin{align}  \label{inequ:nonlinear-frac-max-lip-4}
  \sup_{B}  \dfrac{1}{|B|^{\beta/Q}} \left(  \dfrac{1}{|B|} \dint_{B}  |b(x) -|B|^{-\alpha/Q}M_{\alpha,B}(b)(x)  |^{s} \mathd x \right)^{1/s}  < \infty.
\end{align}
    \label{enumerate:nonlinear-frac-max-lip-4}
   \item  For all $s \in   [1,\infty)$  such that \labelcref{inequ:nonlinear-frac-max-lip-4} holds.
 \label{enumerate:nonlinear-frac-max-lip-5}
\end{enumerate}
\end{theorem}
\begin{remark}   \label{rem.nonlinear-frac-max-lip}
\begin{enumerate}[ label=(\roman*)]
\item For the case $\alpha=0$, the equivalence of \labelcref{enumerate:nonlinear-frac-max-lip-1}, \labelcref{enumerate:nonlinear-frac-max-lip-2} and \labelcref{enumerate:nonlinear-frac-max-lip-4}  was proved   in \cite{wu2023characterizationlip} (see Theorem 1.3).
\item  Moreover, it was proved in Theorem 1.3 of  \cite{wu2023characterizationlip}, see also \cref{lem:non-negative-max-lip} below, that  $b\in  \Lambda_{\beta}(\mathbb{G})$  and $b\ge 0$ if and only if
\begin{align} \label{inequ:1.3-wu2022characterization}
\sup_{B} |B|^{-\beta/Q}  \left( |B|^{-1} \dint_{B}  |b(x) -M_{B}(b)(x)  |^{q} \mathd x \right)^{1/q}    <\infty.
\end{align}
Compared with \labelcref{inequ:1.3-wu2022characterization},  \labelcref{inequ:nonlinear-frac-max-lip-4} gives a new characterization for nonnegative Lipschitz functions.
\end{enumerate}
\end{remark}

Next, we consider some necessary and sufficient conditions for the   boundedness of $M_{\alpha,b}$  when   $b$ belongs to a Lipschitz space.

\begin{theorem} \label{thm:frac-max-lip}
Let $0 <\beta <1$, $0 <\alpha  <\alpha+\beta <Q$ and $b$ be a locally integrable function  on $\mathbb{G}$. Then the following statements are equivalent:
\begin{enumerate}[label=(B.\arabic*)] 
  \item   $b\in \Lambda_{\beta}(\mathbb{G})$.
    \label{enumerate:thm-frac-max-lip-1}
   \item $ M_{\alpha,b} $ is bounded from $L^{p}(\mathbb{G})$ to $L^{q}(\mathbb{G})$ for all $p, q$ with $1<p<\frac{Q}{\alpha+\beta} $ and $\frac{1}{q}  = \frac{1}{p}  -\frac{\alpha+\beta}{Q}$.
    \label{enumerate:thm-frac-max-lip-2}
\item $ M_{\alpha,b} $ is bounded from $L^{p}(\mathbb{G})$ to $L^{q}(\mathbb{G})$ for some $p, q$ with $1<p<\frac{Q}{\alpha+\beta} $ and $\frac{1}{q}  = \frac{1}{p}  -\frac{\alpha+\beta}{Q}$.
    \label{enumerate:thm-frac-max-lip-3}
   \item There exists $s\in   [1,\infty)$ such that   
\begin{align} \label{inequ:frac-max-lip-4}
 \sup_{B}  \dfrac{1}{|B|^{\beta/Q}} \Big( \dfrac{1}{|B|} \dint_{B} |b(x)-b_{B}|^{s} \mathd x \Big)^{1/s} &< \infty.
\end{align}
    \label{enumerate:thm-frac-max-lip-4}
  \item   \labelcref{inequ:frac-max-lip-4} holds for  all $s\in   [1,\infty)$. 
    \label{enumerate:thm-frac-max-lip-5}
\end{enumerate}
\end{theorem}
\begin{remark}
\begin{enumerate}[label=(\roman*)]  
\item   The equivalence of \labelcref{enumerate:thm-frac-max-lip-1}, \labelcref{enumerate:thm-frac-max-lip-2} and \labelcref{enumerate:thm-frac-max-lip-3}  was proved in \cite{wu2023characterizationlip} (see Theorem 1.1 for $\alpha=0$). The equivalence of \labelcref{enumerate:thm-frac-max-lip-1}, \labelcref{enumerate:thm-frac-max-lip-4} and \labelcref{enumerate:thm-frac-max-lip-5} is contained in \cref{lem:2.2-li2003lipschitz} below.
\item When $\mathbb{G} =\mathbb{R}^{n}$, the  above  equivalence  was proved in \cite{zhang2019some} (see Corollary 1.3).
 \item  For the case $\alpha= 0$ and $\mathbb{G} =\mathbb{R}^{n}$, similar results were given in \cite{zhang2017characterization}  for Lebesgue spaces with constant exponents, and in \cite{zhang2019some,zhang2019characterization}  for the variable case.
\end{enumerate}
\end{remark}


This paper is organized as follows. In  \cref{sec:preliminary}, we will recall some basic definitions and known results. In  \cref{sec:proof-mab}, we will prove  main results.

Throughout this paper, the letter $C$  always stands for a constant  independent of the main parameters involved and whose value may differ from line to line.
In addition, we  give some notations. Here and hereafter $L^{p} ~(1\le p\le \infty)$ 	 will always denote the  standard $L^{p} $-space with respect to the Haar measure $\mathd x$, with the $L^{p} $-norm $\|\cdot\|_{p}$. And let $WL^{p}$ be weak-type  $L^{p} $-space. Denote by $\chi_{E}$  the  characteristic function of a measurable set $E$ of $\mathbb{G}$.

\section{Preliminaries and lemmas}
\label{sec:preliminary}

To prove the main results of this paper, we first recall some necessary notions and remarks.
Firstly, we recall some preliminaries concerning stratified Lie groups (or so-called
Carnot groups). We refer the reader to  \cite{folland1982hardy,bonfiglioli2007stratified,stein1993harmonic}.

\subsection{Lie group $\mathbb{G}$}

 \begin{definition}
 \label{def:stratified-Lie-algebra-krantz1982lipschitz}
  Let $m\in \mathbb{Z}^{+}$, $\mathcal{G}$  be a finite-dimensional Lie algebra,  $[X, Y] = XY - YX \in \mathcal{G}$ be Lie bracket with $X,Y  \in \mathcal{G}$.
\begin{enumerate}[label=(\roman*)]
\item If $Z \in \mathcal{G}$ is an $m^{\text{th}}$  order Lie bracket and $W \in \mathcal{G}$, then $[Z,W]$ is an $(m + 1)^{\text{st}}$ order  Lie bracket.
\item  We say $\mathcal{G}$  is a nilpotent Lie algebra of step  $m$ if  $m$ is the smallest integer for which all Lie brackets of order $m+1$ are zero.
\item   We say that  a   Lie algebra $\mathcal{G}$   is  stratified  if there is a direct sum vector space decomposition
\begin{align}\label{equ:lie-algebra-decomposition}
 \mathcal{G} =\oplus_{j=1}^{m} V_{j}  = V_{1} \oplus  \cdots \oplus  V_{m}
\end{align}
such that $\mathcal{G}$ is nilpotent of step $m$, that is,
\begin{align*}
 [V_{1},V_{j}] =
 \begin{cases}
 V_{j+1} &  1\le j \le  m-1  \\
 0 & j\ge m
\end{cases}
\end{align*}
 holds.
\end{enumerate}
\end{definition}

It is not difficult to find that the above $V_{1}$  generates the whole of the Lie algebra $\mathcal{G}$ by taking Lie brackets since each element of $ V_{j}~(2\le j \le m)$  is a linear combination of $(j-1)^{\text{th}}$ order Lie bracket  of elements of $ V_{1}$.

With the help of the related  notions of Lie algebra (see \cref{def:stratified-Lie-algebra-krantz1982lipschitz}), the following definition can be obtained.
 \begin{definition}\label{def:stratified-Lie-group}
  Let $\mathbb{G}$ be a finite-dimensional, connected and simply-connected Lie group   associated with Lie algebra $\mathcal{G}$. Then
\begin{enumerate}[label=(\roman*)]  
\item  $\mathbb{G}$  is called nilpotent if its Lie algebra $\mathcal{G}$ is nilpotent.
\item  $\mathbb{G}$ is said to be stratified if its Lie algebra $\mathcal{G}$ is stratified.
\item  $\mathbb{G}$ is called homogeneous if it is a nilpotent Lie group whose Lie algebra $\mathcal{G}$ admits a family of dilations $\{\delta_{r}\}$, namely, for $r>0$, $X_{k}\in V_{k}~(k=1,\ldots,m)$,
\begin{align*}
  \delta_{r} \Big( \sum_{k=1}^{m} X_{k} \Big)  =  \sum_{k=1}^{m} r^{k} X_{k},
\end{align*}
which are Lie algebra automorphisms.
\end{enumerate}
\end{definition}

\begin{remark}  \label{rem:lie-algebra-decom-zhu2003herz}  %
Let $\mathcal{G} =  \mathcal{G}_{1}\supset  \mathcal{G}_{2} \supset \cdots \supset  \mathcal{G}_{m+1} =\{0\}$   denote the lower central series of  $\mathcal{G}$, and $X=\{X_{1},\dots,X_{n}\}$ be a basis for $V_{1}$ of $\mathcal{G}$.
\begin{enumerate}[label=(\roman*) ]  
\item  (see \cite{zhu2003herz})  The direct sum decomposition  \labelcref{equ:lie-algebra-decomposition} can be constructed by identifying each $\mathcal{G}_{j}$ as a vector subspace of $\mathcal{G}$ and setting $ V_{m}=\mathcal{G}_{m}$ and $ V_{j}=\mathcal{G}_{j}\setminus \mathcal{G}_{j+1}$ for $j=1,\ldots,m-1$.
\item   (see \cite{folland1979lipschitz}) The number $Q=\trace A =\sum\limits_{j=1}^{m} j\dim(V_{j})$ is called the homogeneous dimension of $\mathcal{G}$, where $A$ is a diagonalizable linear transformation of  $\mathcal{G}$ with positive eigenvalues.
\item  (see \cite{zhu2003herz} or \cite{folland1979lipschitz})   
The number  $Q$ is also called the homogeneous dimension of $\mathbb{G}$  since $\mathd(\delta_{r}x)=r^{Q}\mathd x$ for all $r>0$, and
\begin{align*}
 Q = \sum_{j=1}^{m} j \dim(V_{j}) = \sum_{j=1}^{m} \dim(\mathcal{G}_{j}).
\end{align*}
\end{enumerate}
\end{remark}

By the Baker-Campbell-Hausdorff formula for sufficiently small elements $X$ and $Y$ of $\mathcal{G}$ one has
\begin{align*}
 \exp X \exp Y=  \exp H(X,Y)= X+Y +\frac{1}{2}[X,Y]+\cdots
\end{align*}
where $\exp : \mathcal{G} \to \mathbb{G}$ is the exponential map, $H(X, Y )$ is an infinite linear
combination of $X$ and $Y$ and their Lie brackets, and the dots denote terms of order higher than two.  And the above equation is finite in the case  of  $\mathcal{G}$ is a nilpotent Lie algebra.

The following properties can be found in \cite{ruzhansky2019hardy}(see Proposition 1.1.1,  or Proposition 1.2 in \cite{folland1982hardy}).

\begin{proposition}\label{pro:2.1-yessirkegenov2019}
 Let $\mathcal{G}$ be a nilpotent Lie algebra, and let $\mathbb{G}$ be the corresponding connected and simply-connected nilpotent Lie group. Then we have
\begin{enumerate}[label=(\roman*) ]  
\item   The exponential map  $\exp: \mathcal{G} \to \mathbb{G}$  is a diffeomorphism. Furthermore, the group law $(x,y) \mapsto xy$ is a polynomial map if  $\mathbb{G}$ is identified with $\mathcal{G}$ via $\exp$.
\item  If $\lambda$ is a Lebesgue measure on  $\mathcal{G}$, then $\exp\lambda$ is a bi-invariant Haar measure on  $\mathbb{G}$ (or a bi-invariant Haar measure $\mathd  x$ on  $\mathbb{G}$  is just the lift of Lebesgue measure on  $\mathcal{G}$ via $\exp$).
\end{enumerate}
\end{proposition}

Thereafter, we use $Q$ to denote the homogeneous dimension of  $\mathbb{G}$, $y^{-1}$ represents the inverse of $y\in \mathbb{G}$,  $y^{-1}x$ stands for the group multiplication of $y^{-1}$  by $x$ and the group identity  element of $\mathbb{G}$ will be referred to as the origin denotes by $e$.

A homogenous norm on $\mathbb{G}$ is a continuous function $x\to \rho(x)$  from $\mathbb{G}$ to $[0,\infty)$, which is  $C^{\infty}$ on $\mathbb{G}\setminus\{0\}$ and satisfies
\begin{align*}
\begin{cases}
 \rho(x^{-1}) =  \rho(x), \\
 \rho(\delta_{t}x) =  t\rho(x) \ \ \text{for all}~  x \in \mathbb{G} ~\text{and}~ t > 0, \\
 \rho(e) =  0.
\end{cases}
\end{align*}
Moreover, there exists a constant $c_{0} \ge 1$ such that $\rho(xy) \le c_{0}(\rho(x) + \rho(y))$ for all $x,y \in \mathbb{G}$.

 With the norm above, we define the $\mathbb{G}$ ball centered at $x$ with radius $r$ by $B(x, r) = \{y \in \mathbb{G} : \rho(y^{-1}x) < r\}$,  and by $\lambda B$ denote the ball $B(x,\lambda r)$  with $\lambda>0$, let $B_{r} = B(e, r) = \{y \in \mathbb{G}  : \rho(y) < r\}$ be the open ball centered at $e$ with radius $r$,  which is the image under $\delta_{r}$ of $B(e, 1)$.
 And by $\sideset{^{\complement}}{}  {\mathop {B(x,r)}} = \mathbb{G}\setminus B(x,r)= \{y \in \mathbb{G} : \rho(y^{-1}x) \ge r\}$ denote the complement of $B(x, r)$.  Let  $|B(x,r)|$ be the Haar measure of the ball  $B(x,r)\subset \mathbb{G}$, and
 there exists $c_{1} =c_{1} (\mathbb{G})$ such that
\begin{align*}
  |B(x,r)| = c_{1} r^{Q}, \ \  \ \   x\in \mathbb{G}, r>0.
\end{align*}
In addition,  the Haar measure of a homogeneous Lie group  $\mathbb{G}$  satisfies the doubling condition (see pages 140 and 501,\cite{fischer2016quantization}), i.e. $\forall ~ x\in \mathbb{G}$, $r>0$, $\exists~ C$, such that
\begin{align*}
   |B(x,2r)| \le C |B(x,r)|.
\end{align*}

The most basic partial differential operator in a stratified Lie group is the sub-Laplacian associated with $X=\{X_{1},\dots,X_{n}\}$, i.e., the second-order partial differential operator on  $\mathbb{G}$  given by
\begin{align*}
 \mathfrak{L} =  \sum_{i=1}^{n} X_{i}^{2}
\end{align*}

The following lemma is known as the  H\"{o}lder's inequality on Lebesgue spaces over Lie groups $\mathbb{G}$, it can also be found in  \cite{guliyev2022some}, when  Young function $\Phi(t)=t^{p}$ and its complementary function $\Psi(t)=t^{q}$ with $\frac{1}{p}+\frac{1}{q}=1$.
\begin{lemma}[H\"{o}lder's inequality on  $\mathbb{G}$]\label{lem:holder-inequality-Lie-group}
Let $1\le p,q \le\infty$ with $\frac{1}{p}+\frac{1}{q}=1$, $\Omega\subset  \mathbb{G}$ be a measurable set  and measurable functions $f\in L^{p}(\Omega)$ and $g\in L^{q}(\Omega)$.  Then there exists a positive constant $C$ such that
\begin{align*} 
   \dint_{\Omega} |f(x)g(x)|  \mathrm{d}x \le C \|f\|_{L^{p}(\Omega)} \|g\|_{L^{q}(\Omega)}.
\end{align*}
\end{lemma}

By elementary calculations we have the following property. It can also  be found in \cite{guliyev2022some}, when  Young function $\Phi(t)=t^{p}$.
\begin{lemma}[Norms of characteristic functions]\label{lem:norm-characteristic-functions-Lie-group}
Let $0<p<\infty$ and $\Omega\subset  \mathbb{G}$ be a measurable set with finite Haar measure. Then
\begin{align*} 
  \|\dchi_{\Omega}\|_{L^{p}(\mathbb{G})} = \|\dchi_{\Omega}\|_{WL^{p}(\mathbb{G})}  = |\Omega|^{1/p}.
\end{align*}
\end{lemma}

\subsection{Lipschitz spaces on $\mathbb{G}$}

Next we give the definition of the Lipschitz spaces on $\mathbb{G}$, and state some basic properties and useful lemmas.

\begin{definition}[Lipschitz-type spaces on $\mathbb{G}$]   \label{def.lip-space} \
\begin{enumerate}[ label=(\roman*)]
\item   Let $0<\beta <1$, we say a function $b$ belongs to the Lipschitz space $\Lambda_{\beta}(\mathbb{G}) $ if there exists a constant $C>0$ such that for all  $x,y\in \mathbb{G}$,
\begin{align*}
  |b(x)-b(y)|   &\le C(\rho(y^{-1}x))^{\beta},
\end{align*}
where $\rho$ is the homogenous norm. The smallest such constant $C$ is called the $\Lambda_{\beta}$  norm of $b$ and is denoted by $\|b\|_{\Lambda_{\beta}(\mathbb{G})}$.
    \label{enumerate:def-lip-1}
\item (see \cite{macias1979lipschitz} ) Let $0<\beta <1$ and $1\le p<\infty$.   The space $\lip_{\beta,p}(\mathbb{G}) $ is defined to be the set of all locally integrable  functions $b$, i.e., there exists a positive constant $C $, such that
\begin{align*}
      \sup_{B\ni x} \dfrac{1}{ |B|^{\beta/Q}}\Big( \dfrac{1}{|B|}  \dint_{B} |b(x)- b_{B}|^{p}\mathd x \Big)^{1/p} \le C
\end{align*}
where the supremum is taken over every ball $B\subset \mathbb{G}$ containing $x$ and $b_{B}=\frac{1}{|B|} \int_{B} b(x) \mathd x$. The least constant $C$   satisfying the conditions above shall   be denoted by $\|b\|_{\lip_{\beta,p}(\mathbb{G})}$.
    \label{enumerate:def-lip-2}
\end{enumerate}
\end{definition}

\begin{remark}  \label{rem.Lipschitz-def}
\begin{enumerate}[label=(\roman*)]
\item  Similar to the definition of Lipschitz space $\Lambda_{\beta}(\mathbb{G}) $ in \labelcref{enumerate:def-lip-1}, we also have the definition form as following  (see  \cite{krantz1982lipschitz,chen2010lipschitz,fan1995characterization} et al.)
\begin{align*}
 \|b\|_{\Lambda_{\beta}(\mathbb{G})}&= \sup_{x,y\in \mathbb{G}\atop y\neq e} \dfrac{|b(xy)- b(x)|}{(\rho(y))^{\beta}}   = \sup_{x,y\in \mathbb{G} \atop x\neq y} \dfrac{|b(x)-b(y)|}{(\rho(y^{-1}x))^{\beta}}.
\end{align*}
And $\|b\|_{\Lambda_{\beta}(\mathbb{G})} =0$   if and only if $b$ is constant.
\item  In \labelcref{enumerate:def-lip-2},  when   $p=1$, we have
\begin{align*}
     \|b\|_{\lip_{\beta,1}(\mathbb{G})} =\sup_{B\ni x} \dfrac{1}{ |B|^{\beta/Q}}\Big( \dfrac{1}{|B|}  \dint_{B} |b(x)- b_{B}| \mathd x \Big) :=\|b\|_{\lip_{\beta}(\mathbb{G})}
\end{align*}
\end{enumerate}
\end{remark}

\begin{lemma} (see \cite{macias1979lipschitz,chen2010lipschitz,li2003lipschitz} ) \label{lem:2.2-li2003lipschitz}
Let   $0<\beta<1$ and the function $b(x)$ integrable on bounded subsets of $\mathbb{G}$.
\begin{enumerate}[label=(\roman*)]
\item  When $1\le p<\infty$,  then
\begin{align*}
 \|b\|_{\Lambda_{\beta}(\mathbb{G})} &=  \|b\|_{\lip_{\beta}(\mathbb{G})} \approx  \|b\|_{\lip_{\beta,p}(\mathbb{G})}.
\end{align*}
\item   Let balls $B_{1}\subset B_{2}\subset \mathbb{G}$ and $b\in \lip_{\beta,p}(\mathbb{G})$ with $p\in [1,\infty]$. Then there exists a constant $C$ depends on $B_{1}$ and $B_{2}$ only, such that
\begin{align*}
     |b_{B_{1}}- b_{B_{2}} |   &\le    C  \|b\|_{\lip_{\beta,p}(\mathbb{G})} |B_{2}|^{\beta/Q}
\end{align*}
\item   When $1\le p<\infty$, then there exists a constant $C$ depends on $\beta$ and $p$ only, such that
\begin{align*}
     | b(x)-  b(y) |   &\le   C  \|b\|_{\lip_{\beta,p}(\mathbb{G})} |B|^{\beta/Q}
\end{align*}
holds for any ball $B$ containing $x$ and $y$.
\end{enumerate}
\end{lemma}

\subsection{Some pointwise estimates and auxiliary  lemmas}

Hereafter, for a function $b$ defined on $\mathbb{G}$, we denote
\begin{align*}
  b^{-}(x) :=- \min\{b, 0\} =
\begin{cases}
 0,  & \text{if}\ b(x) \ge 0  \\
 |b(x)|, & \text{if}\ b(x) < 0
\end{cases}
\end{align*}
and  $b^{+}(x) =|b(x)|-b^{-}(x)$. Obviously, $b(x)=b^{+}(x)-b^{-}(x)$.

From the proof of Theorem 1.3 in \cite{wu2023characterizationlip}, we can obtain the following characterization of nonnegative Lipschitz functions.
\begin{lemma}  \label{lem:non-negative-max-lip}
 Let $0 <\beta <1$ and $b$ be a locally integrable function on $\mathbb{G}$. Then the following assertions are equivalent:
\begin{enumerate}[label=(\roman*)]
\item   $b\in  \Lambda_{\beta}(\mathbb{G})$  and $b\ge 0$.
     \label{enumerate:Lem-non-negative-max-lip-1}
   \item For all $1\le s<\infty$,  there exists a positive constant $C$ such that
\begin{align} \label{inequ:non-negative-max-lip}
 \sup_{B} |B|^{-\beta/Q}  \left( |B|^{-1} \dint_{B}  |b(x) -M_{B}(b)(x)  |^{s} \mathd x \right)^{1/s} \le C.
\end{align}
    \label{enumerate:Lem-non-negative-max-lip-2}
   \item \labelcref{inequ:non-negative-max-lip} holds for some $1\le s<\infty$.
     \label{enumerate:Lem-non-negative-max-lip-3}
\end{enumerate}
\end{lemma}

\begin{proof}

 Since the implication \labelcref{enumerate:Lem-non-negative-max-lip-2}  $\xLongrightarrow{\ \  }$ \labelcref{enumerate:Lem-non-negative-max-lip-3} follows readily, and the implication \labelcref{enumerate:Lem-non-negative-max-lip-3}  $\xLongrightarrow{\ \  }$ \labelcref{enumerate:Lem-non-negative-max-lip-1} was
proved in \cite[Theorem 1.3]{wu2023characterizationlip}, we only need to prove \labelcref{enumerate:Lem-non-negative-max-lip-1}  $\xLongrightarrow{\ \  }$ \labelcref{enumerate:Lem-non-negative-max-lip-2}.

If  $b\in  \Lambda_{\beta}(\mathbb{G})$  and $b\ge 0$, then it follows from \cite[Theorem 1.3]{wu2023characterizationlip} that \labelcref{inequ:non-negative-max-lip} holds for all $s$ with $Q/(Q-\beta)<s<\infty$. Applying H\"{o}lder's inequality, we see that \labelcref{inequ:non-negative-max-lip}  holds for $1\le s\le Q/(Q-\beta)$ as well.

So, the implication \labelcref{enumerate:Lem-non-negative-max-lip-1}  $\xLongrightarrow{\ \  }$ \labelcref{enumerate:Lem-non-negative-max-lip-2} is proven.

\end{proof}

The following strong-type estimate for the fractional maximal function   $M_{\alpha}$  is well known, which can be obtained from
\cite[Proposition A]{kokilashvili1989fractional} or \cite[Theorem 1.6]{bernardis1994two} when  the weights  are constant 1,  see  \cite{macias1981well}, \cite{kokilashvili1989fractional} or \cite{bernardis1994two} for more details.

\begin{lemma}\label{lem:frac-maximal-kokilashvili1989fractional}
Let   $0<\alpha<Q$, $1< p< Q/\alpha$ and $1/q=1/p-\alpha/Q$. If $f\in L^{p}(\mathbb{G})$.
  then there exists a positive constant $C$ such that
\begin{align*}
 \|M_{\alpha}(f)\|_{L^{q}(\mathbb{G})} &\le C  \|f\|_{L^{p}(\mathbb{G})}.
\end{align*}

\end{lemma}

\begin{remark}   \label{rem.a.e.-frac-maximal}
 \begin{enumerate}[label=(\roman*)]
\item  By    \cref{lem:frac-maximal-kokilashvili1989fractional}, if  $0 <\alpha<Q$,  $1< p< Q/\alpha$   and $f\in L^{p}(\mathbb{G})$, then $M_{\alpha}(f)(x)<\infty $ for almost everywhere  $x\in \mathbb{G}$.
\item  The above lemma  can also refer to Theorem 3.3 in \cite{guliyev2022some} when Young function $\Phi(t) = t^{p}$  and its complementary function $\Psi(t)=t^{q}$ with $1/q=1/p-\alpha/Q$.
\end{enumerate}
\end{remark}

Now, we give the following pointwise estimate for $[b,M_{\alpha}] $  on $\mathbb{G}$ when $b\in \Lambda_{\beta}(\mathbb{G})$.
\begin{lemma} \label{lem:frac-maximal-pointwise}
Let   $0\le\alpha<Q$, $0<\beta <1$, $0<\alpha+\beta<Q$ and $f: \mathbb{G} \to \mathbb{R}$ be a locally integrable function.  If  $b\in \Lambda_{\beta}(\mathbb{G})$ and $b\ge 0$, then, for arbitrary   $x\in \mathbb{G} $ such that $M_{\alpha} (f)(x) <\infty$, we have
\begin{align*}
  \big|[b,M_{\alpha}] (f)(x)\big|  &\le \|b\|_{\Lambda_{\beta}(\mathbb{G})} M_{\alpha+\beta} (f)(x).
\end{align*}
\end{lemma}

\begin{proof}
 Similar to the discussion of lemma 2.11 in \cite{zhang2019some}. For any fixed $x \in \mathbb{G}$ such that $M_{\alpha}(f)(x) <\infty$,    if  $b\in \Lambda_{\beta}(\mathbb{G})$ and $b\ge 0$, then we have
\begin{align*}
\big|[b,M_{\alpha}] (f)(x) \big| &= \big|b(x)M_{\alpha}(f)(x)-M_{\alpha}(bf)(x) \big|  \\
 &= \bigg| \sup_{B\ni x \atop B\subset \mathbb{G}}  \dfrac{1}{|B|^{1-\alpha/Q}}  \dint_{B} b(x)|f(y)|  \mathd y  \\
 &\;\qquad  -\sup_{B\ni x \atop B\subset \mathbb{G}} \dfrac{1}{|B|^{1-\alpha/Q}}  \dint_{B} b(y)|f(y)|  \mathd y  \bigg| \\
 &\le \sup_{B\ni x \atop B\subset \mathbb{G}}  \dfrac{1}{|B|^{1-\alpha/Q}}  \dint_{B} |b(x)-b(y)| |f(y)| \mathd y  \\
  &\le \|b\|_{\Lambda_{\beta}(\mathbb{G})} \sup_{B\ni x \atop B\subset \mathbb{G}}  \dfrac{1}{|B|^{1-(\alpha+\beta)/Q}}  \dint_{B} |f(y)| \mathd y  \\
&\le \|b\|_{\Lambda_{\beta}(\mathbb{G})} M_{\alpha+\beta} (f)(x).
\end{align*}
\end{proof}

Similar to   Lemma 2.3 in \cite{zhang2009commutators}, we get the following result.
\begin{lemma} \label{lem:frac-maximal-pointwise-relation}
 Let  $0 \le \alpha<Q$, $B\subset \mathbb{G}$ be a ball, and $f$ be a locally integrable function.     Then, for all $x\in B$, we have
\begin{align} \label{equ:frac-maximal-pointwise-relation}
   M_{\alpha} (f\dchi_{B})(x)  &=  M_{\alpha,B}(f)(x).
\end{align}
\end{lemma}

\begin{proof}
 Some ideas are taken from \cite{bastero2000commutators} and \cite{zhang2009commutators}. Reasoning as the discussion of lemma 2.3 in \cite{zhang2009commutators}. For any  $x \in B$, it is easy to verify that
\begin{align} \label{inequ:frac-maximal-pointwise-relation-r}
  M_{\alpha} (f\dchi_{B})(x)   \ge  M_{\alpha,B}(f)(x)
\end{align}
  from the definitions of $M_{\alpha} (f\dchi_{B})(x)$ and $M_{\alpha,B}(f)(x)$.

  So, in order to prove the equality \labelcref{equ:frac-maximal-pointwise-relation} is true, we only need to prove the following realtion, namely,  for any  $\mathbb{G}$-ball $B^{*}\ni x$ with radius $r^{*}$, there exist $\mathbb{G}$-ball $B'\ni x$  with radius $r'$ and $B'\subset B$, such that
\begin{align} \label{inequ:frac-maximal-pointwise-relation-l}
  \dfrac{1}{|B^{*}|^{1-\alpha/Q}} \dint_{B^{*}}  |f(y)\dchi_{B}(y)| \mathd y   \le   \dfrac{1}{|B'|^{1-\alpha/Q}} \dint_{B'} |f(y)| \mathd y.
\end{align}

Indeed, for the case $B^{*} \cap B =\emptyset$, it is  clear that \labelcref{inequ:frac-maximal-pointwise-relation-l} is true since $f(y)\dchi_{B}(y)=0$ for any $y\in B^{*}$.

Now we divide   $B^{*} \cap B \neq\emptyset$ into two  cases to consider.
 \begin{enumerate}[label=(\alph*)]
\item When the relation between $B^{*}$ and $B$ is inclusion.
Without loss of generality, let  $B^{*} \supset B$, then \labelcref{inequ:frac-maximal-pointwise-relation-l} is valid when we take $B'=B =B^{*} \cap B$.
\item  When $B^{*}\not\subset  B$ and $B\not\subset  B^{*}$, we consider the relation between $|B|$ and $ |B^{*}|$.
 \begin{enumerate}[label=(\roman*)]
\item  Assume $|B|\le |B^{*}|$. Then we may take $B'=B \supset B^{*} \cap B$, so \labelcref{inequ:frac-maximal-pointwise-relation-l} is   true.
\item    Assume $|B|> |B^{*}|$. Firstly, since $B^{*} \cap B$ is a bounded set in $\mathbb{G}$ and $x\in B^{*} \cap B$, then there exists  not only a   minimal  ball  $B'''$ containing the intersection $  B^{*} \cap B$
but also a maximal   ball $B''\subset B^{*} \cap B$ containing $x$   inscribed in the ball $B$  at a point $P$, namely,  $x\in B''\subset B^{*} \cap B$, $\partial B''\cap \partial B =\{P\}$, $B^{*}\cap B\subset B'''$ and $|B'''|\le |B^{*}|$.
Indeed,  when the spherical center of $B^{*} $ belongs to $B^{*} \cap B$, we can take $B'''=B^{*}$, otherwise $|B'''|< |B^{*}|$.

Secondly, there is a ball $B'\subset B$ such that $x\in B''\subset  B'$, $\partial B'\cap \partial B =\{P\}$ and $  |B'|= |B'''|\le |B^{*}|$. Let $B'=(B'\cap B^{*})\cup  (B'\setminus B^{*} )$ and $B^{*} \cap B = (B'\cap B^{*})\cup \big((B^{*} \cap B)\setminus B'\big)$ satisfy $x\in B''\subset B'\cap B^{*}$ and $x\not\in (B'\setminus B^{*} )\cup \big((B^{*} \cap B)\setminus B'\big)$.

Furthermore, for a given ball $B $, $f \dchi _ { B } $ is integrable and finite almost everywhere since $f$ is a locally integrable function.
Observe the fact that $B'\setminus B^{*}$  is larger than $(B^{*} \cap B)\setminus B'$,
and neither contains $x$. Then there is an $\Omega\subset B'\setminus B^{*}$ such that      $\int_{(B^{*} \cap B)\setminus B'}  |f(y) | \mathd y \le  \int_{\Omega}  |f(y)| \mathd y$.

Combined with the discussion above, it follows that
\begin{align*}
\dfrac{1}{|B^{*}|^{1-\alpha/Q}} \dint_{B^{*}}  |f(y)\dchi_{B}(y)| \mathd y
 &=  \Big( \dfrac{|B'|}{|B^{*}|}\Big)^{1-\alpha/Q} \dfrac{1}{|B'|^{1-\alpha/Q}} \dint_{B^{*}\cap B}  |f(y)\dchi_{B}(y)| \mathd y  \\
 &\le  \dfrac{1}{|B'|^{1-\alpha/Q}} \Big( \dint_{B'\cap B^{*}}  |f(y)\dchi_{B}(y)| \mathd y +  \dint_{(B^{*} \cap B)\setminus B'}  |f(y)\dchi_{B}(y)| \mathd y  \Big) \\
 &\le  \dfrac{1}{|B'|^{1-\alpha/Q}} \Big( \dint_{B'\cap B^{*}}  |f(y)\dchi_{B}(y)| \mathd y +  \dint_{\Omega}  |f(y)\dchi_{B}(y)| \mathd y  \Big) \\
 &\le  \dfrac{1}{|B'|^{1-\alpha/Q}} \Big( \dint_{B'\cap B^{*}}  |f(y)\dchi_{B}(y)| \mathd y +  \dint_{B'\setminus B^{*} }  |f(y)\dchi_{B}(y)| \mathd y  \Big) \\
 &= \dfrac{1}{|B'|^{1-\alpha/Q}} \dint_{B'}  |f(y)| \mathd y .
\end{align*}
\end{enumerate}

\end{enumerate}
 
Summarizing the discussion above we find that \labelcref{inequ:frac-maximal-pointwise-relation-l} is valid.

\labelcref{inequ:frac-maximal-pointwise-relation-r} and \labelcref{inequ:frac-maximal-pointwise-relation-l} together give \labelcref{equ:frac-maximal-pointwise-relation},  this completes the proof.
\end{proof}

\begin{remark}   \label{rem.frac-maximal-pointwise-relation}
 \begin{enumerate}[label=(\roman*)]
\item Further, by applying   \cref{lem:frac-maximal-pointwise-relation} and the definition of $M_{\alpha,B}(\dchi_{B})(x)$, we have that
\begin{align*}
    M_{\alpha} (\dchi_{B})(x)  &=  M_{\alpha,B}(\dchi_{B})(x)=|B|^{\alpha/Q}.
\end{align*}
\item  For the case $\alpha=0$, the following results are also valid, namely
\begin{align*}
    M (\dchi_{B})(x)  &=  M_{B}(\dchi_{B})(x)=\dchi_{B}(x), \ \ M  (f\dchi_{B})(x)  = M_{B}(f)(x).
\end{align*}
\end{enumerate}
\end{remark}

 Referring  to   \cite[page 3331]{bastero2000commutators} or \cite{zhang2009commutators},   through elementary calculations and derivations, it is easy to check that the following assertions are true. 
\begin{lemma} \label{lem:frac-max-pointwise-assert}
Let  $b$ be a locally integrable function on $\mathbb{G}$ and $B \subset \mathbb{G}$ be  an arbitrary  given    ball.
 \begin{enumerate}[label=(\roman*)]
\item  If $E=\{x\in B: b(x)\le b_{B}\}$ and $F=  B\setminus E =\{x\in B: b(x)> b_{B}\}$. Then the following equality
\begin{align*}
    \dint_{E} |b(x)-b_{B}| \mathd x  &=  \dint_{F} |b(x)-b_{B}| \mathd x
\end{align*}
  is trivially true.
\item    Then for any  $x\in B$, we have
\begin{align*}
    |b_{B}|   &\le  |B|^{-\alpha/Q} M_{\alpha,B}(b)(x).
\end{align*}

\end{enumerate}
\end{lemma}

\section{Proof of  the principal results } 
\label{sec:proof-mab}

We now give the proof of the  principal results. 

 \subsection{Proof of \cref{thm:nonlinear-frac-max-lip}}

To   prove \cref{thm:nonlinear-frac-max-lip}, we first prove the following lemma.

\begin{lemma} \label{lem:frac-Lie-lip-norm}
Let  $0 <\beta <1$ and $0 <\alpha <Q$. If   $b$  is a locally integrable function on $\mathbb{G}$ and satisfies
\begin{align} \label{inequ:lem-frac-Lie-lip-norm}
 \sup_{B} \dfrac{1}{|B|^{\beta/Q}} \left(  \dfrac{1}{|B|} \dint_{B}  |b(x) -|B|^{-\alpha/Q}M_{\alpha,B}(b)(x)  |^{s} \mathd x \right)^{1/s}  < \infty
\end{align}
 for some $s\in [1,\infty)$, then $b\in \Lambda_{\beta}(\mathbb{G})$.
\end{lemma}

\begin{proof}
Some ideas are taken from \cite{bastero2000commutators,zhang2009commutators,zhang2014commutators} and  \cite{zhang2019some}.

For any   $\mathbb{G}$-ball $B\subset \mathbb{G}$, let $E=\{x\in B: b(x)\le b_{B}\}$ and $F=  B\setminus E =\{x\in B: b(x)> b_{B}\}$.
Noticing from \cref{lem:frac-max-pointwise-assert}(ii) that
\begin{align*}
    |b_{B}|   &\le  |B|^{-\alpha/Q} M_{\alpha,B}(b)(x) \qquad \forall ~ x\in B.
\end{align*}
Then, for any $x\in E\subset B$, we have $b(x)\le b_{B}\le  |b_{B}| \le   |B|^{-\alpha/Q} M_{\alpha,B}(b)(x)$.
It is clear that
\begin{align*}
    |b(x)- b_{B}|   &\le \Big| b(x) - |B|^{-\alpha/Q} M_{\alpha,B}(b)(x) \Big|, \qquad  \forall~ x\in E.
\end{align*}

Therefore, by using  \cref{lem:frac-max-pointwise-assert}(i),   we get
\begin{align*}
 \dfrac{1}{|B|^{1+\beta/Q}} \dint_{B} \big| b(x)-b_{B}) \big| \mathd x &=  \dfrac{1}{|B|^{1+\beta/Q}} \dint_{E\cup F} \big| b(x)-b_{B}) \big| \mathd x  \\
 &= \dfrac{2}{|B|^{1+\beta/Q}} \dint_{E} \big| b(x)-b_{B}) \big| \mathd x    \\
  &\le \dfrac{2}{|B|^{1+\beta/Q}} \dint_{E} \Big| b(x) - |B|^{-\alpha/Q} M_{\alpha,B}(b)(x) \Big| \mathd x    \\
  &\le  \dfrac{2}{|B|^{1+\beta/Q}} \dint_{B} \Big| b(x) - |B|^{-\alpha/Q} M_{\alpha,B}(b)(x) \Big| \mathd x .
\end{align*}

By using \cref{lem:holder-inequality-Lie-group}, \labelcref{inequ:lem-frac-Lie-lip-norm} and   \cref{lem:norm-characteristic-functions-Lie-group}, we have
\begin{align*}
 \dfrac{1}{|B|^{1+\beta/Q}} & \dint_{B} \big| b(x)-b_{B}) \big| \mathd x   \\
  &\le  \dfrac{2}{|B|^{1+\beta/Q}} \dint_{B} \Big| b(x) - |B|^{-\alpha/Q} M_{\alpha,B}(b)(x) \Big| \mathd x  \\
  &\le \dfrac{C}{|B|^{1+\beta/Q}} \left( \dint_{B}  |b(x) -|B|^{-\alpha/Q}M_{\alpha,B}(b)(x)  |^{s} \mathd x \right)^{1/s}   \|\dchi_{B}\|_{L^{s'}(\mathbb{G}) } \\
  &\le  \dfrac{C}{|B|^{\beta/Q}} \left(  \dfrac{1}{|B|} \dint_{B}  |b(x) -|B|^{-\alpha/Q}M_{\alpha,B}(b)(x)  |^{s} \mathd x \right)^{1/s}   \\
  &\le C.
\end{align*}

So, the proof is completed by applying \cref{lem:2.2-li2003lipschitz} and \cref{def.lip-space}.
\end{proof}

\begin{figure}[!ht] \centering
\scalebox{0.6}{
\begin{tikzpicture}[
  vertex/.style = {shape=circle,draw,minimum size=2em},
  edge/.style = {->,-Latex},
  ]
  \node[vertex] (o) at (0,0) {1};
  \node[vertex] (t) at (-2,-2) {2};
  \node[vertex] (th) at (-2,-5) {3};
  \node[vertex] (f) at (2,-5) {4};
  \node[vertex] (fv) at (2,-2) {5};
  \draw[edge,very thick,dashed] (t) to  node[left] {$w_{23}$} (th);
  \draw[edge,very thick,dashed] (fv) to node[right] {$w_{54}$} (f);

  \draw[edge, very thick,blue] (o) -- (t) node[midway,left] {$w_{12}$}  ;
  \draw[edge, very thick,blue] (th) to node[above, midway]  {$w_{34}$} (f);

  \draw[edge,very thick,blue] (f) to[out=0, in=0] node[above, yshift=11mm] {$w_{41}$} (o);

  \draw[edge,very thick,blue] (t) to node[below, midway]  {$w_{25}$} (fv);
\end{tikzpicture}
}
\vskip 3pt
\caption{Proof structure 
\\ where $w_{ij}$ denotes $i\Longrightarrow j$}\label{fig:ps-equivalent-non}
\end{figure}

\begin{proof}[Proof of \cref{thm:nonlinear-frac-max-lip}]
Since the implications \labelcref{enumerate:nonlinear-frac-max-lip-2} $\xLongrightarrow[]{\ \  }$ \labelcref{enumerate:nonlinear-frac-max-lip-3}  and \labelcref{enumerate:nonlinear-frac-max-lip-5} $\xLongrightarrow[]{\ \  }$ \labelcref{enumerate:nonlinear-frac-max-lip-4}    follows readily,
we only need to prove \labelcref{enumerate:nonlinear-frac-max-lip-1} $\xLongrightarrow[]{\ \  }$ \labelcref{enumerate:nonlinear-frac-max-lip-2},
 \labelcref{enumerate:nonlinear-frac-max-lip-3} $\xLongrightarrow[]{\ \  }$ \labelcref{enumerate:nonlinear-frac-max-lip-4},
 \labelcref{enumerate:nonlinear-frac-max-lip-4} $\xLongrightarrow[]{\ \  }$ \labelcref{enumerate:nonlinear-frac-max-lip-1}, and \labelcref{enumerate:nonlinear-frac-max-lip-2} $\xLongrightarrow[]{\ \  }$ \labelcref{enumerate:nonlinear-frac-max-lip-5}
 (see \Cref{fig:ps-equivalent-non} for the proof structure).


\labelcref{enumerate:nonlinear-frac-max-lip-1} $\xLongrightarrow[]{\ \  }$ \labelcref{enumerate:nonlinear-frac-max-lip-2}:\
Let $b\in \Lambda_{\beta}(\mathbb{G})$ and $b\ge 0$. We need to prove that $[b,M_{\alpha} ]$ is bounded from $L^{p}(\mathbb{G})$ to $L^{q}(\mathbb{G})$ for all $p$ and $q$ satisfy $1 <p < \frac{Q}{\alpha+\beta}$ and $\frac{1}{q} =\frac{1}{p} - \frac{\alpha+\beta}{Q}$.
For such $p$ and any $f\in L^{p}(\mathbb{G})$, it follows from \cref{rem.a.e.-frac-maximal}(i) that  $M_{\alpha}(f)(x)<\infty $ for almost everywhere  $x\in \mathbb{G}$.
By \cref{lem:frac-maximal-pointwise}, we have
\begin{align*}
  \big|[b,M_{\alpha}] (f)(x)\big|  &\le \|b\|_{\Lambda_{\beta}(\mathbb{G})} M_{\alpha+\beta} (f)(x).
\end{align*}
Then, assertion \labelcref{enumerate:nonlinear-frac-max-lip-2} follows from \labelcref{lem:frac-maximal-kokilashvili1989fractional}.

 \labelcref{enumerate:nonlinear-frac-max-lip-3} $\xLongrightarrow[]{\ \  }$ \labelcref{enumerate:nonlinear-frac-max-lip-4}:\
 Let $(p,q)$ be such that $[b,M_{\alpha} ]$ is bounded from $L^{p}(\mathbb{G})$ to $L^{q}(\mathbb{G})$. We will verify \labelcref{inequ:nonlinear-frac-max-lip-4} for $s=q$.

 For any fixed  $\mathbb{G}$-ball $B\subset \mathbb{G}$ and any  $x \in B$, it follows from \cref{lem:frac-maximal-pointwise-relation} and \cref{rem.frac-maximal-pointwise-relation}   that the pointwise estimates
\begin{align*}
   M_{\alpha} (b\dchi_{B})(x)   =   M_{\alpha,B} (b)(x)
 \ \text{and} \
  M_{\alpha} (\dchi_{B})(x)   =   M_{\alpha,B} (\dchi_{B})(x)= |B|^{\alpha/Q}.
\end{align*}
Then, for  any $x \in B$, we have
\begin{align*}
    b(x) -|B|^{-\alpha/Q}M_{\alpha,B}(b)(x)   &=  |B|^{-\alpha/Q} \Big( b(x)|B|^{\alpha/Q} - M_{\alpha,B}(b)(x) \Big)  \\
    &=  |B|^{-\alpha/Q} \Big( b(x) M_{\alpha} (\dchi_{B})(x)  - M_{\alpha} (b\dchi_{B})(x)  \Big)  \\
    &= |B|^{-\alpha/Q} [b,M_{\alpha}] (\dchi_{B})(x).
\end{align*}

Noting that   $ [b,M_{\alpha}] $ is bounded from    $L^{p}(\mathbb{G})$ to $L^{q}(\mathbb{G})$ with $\frac{1}{q}  = \frac{1}{p}  -\frac{\alpha+\beta}{Q}$. For any ball $B\subset \mathbb{G}$,  applying \cref{lem:norm-characteristic-functions-Lie-group}, we obtain
\begin{align*}
\begin{aligned}
 \dfrac{1}{|B|^{\beta/Q}} &\left(  \dfrac{1}{|B|} \dint_{B}  |b(x) -|B|^{-\alpha/Q}M_{\alpha,B}(b)(x)  |^{q} \mathd x \right)^{1/q}   \\
 &\le   |B|^{-(\alpha+\beta)/Q-1/q} \big\|[b,M_{\alpha}] (\dchi_{B}) \big\|_{L^{q}(\mathbb{G})}    \\
  &\le  C |B|^{-(\alpha+\beta)/Q-1/q}   \| \dchi_{B} \|_{L^{p}(\mathbb{G})}  \\
    & \le C,
 \end{aligned}
\end{align*}
which gives  \labelcref{inequ:nonlinear-frac-max-lip-4} for $s=q$ since the ball $B\subset \mathbb{G}$ is arbitrary and $C$ is independent of $B$.

\labelcref{enumerate:nonlinear-frac-max-lip-4} $\xLongrightarrow[]{\ \  }$ \labelcref{enumerate:nonlinear-frac-max-lip-1}:\
By \cref{lem:non-negative-max-lip}, it suffices to prove
\begin{align} \label{inequ:proof-lem-non-negative-max-lip-41}
 \sup_{B} \dfrac{1}{|B|^{1+\beta/Q} }   \dint_{B}  \big|b(x) -M_{B}(b)(x)  \big| \mathd x   <\infty.
\end{align}

For any fixed ball $B\subset \mathbb{G}$, we have
\begin{align} \label{inequ:proof-lem-non-negative-max-lip-41-2}
\begin{aligned}
  \dfrac{1}{|B|^{1+\beta/Q} }  &  \dint_{B}  \big|b(x) -M_{B}(b)(x)  \big| \mathd x  \\
  &\le   \dfrac{1}{|B|^{1+\beta/Q} }   \dint_{B}  \Big| b(x)-|B|^{-\alpha/Q} M_{\alpha,B}(b)(x) \Big| \mathd x   \\
   &\qquad +\dfrac{1}{|B|^{1+\beta/Q} }   \dint_{B}  \Big| |B|^{-\alpha/Q} M_{\alpha,B}(b)(x)  -M_{B}(b)(x)  \Big| \mathd x   \\
    &:=I_{1}+I_{2}.
\end{aligned}
\end{align}

For $ I_{1}$, by applying statement \labelcref{enumerate:nonlinear-frac-max-lip-4},     \cref{lem:holder-inequality-Lie-group} (H\"{o}lder's inequality) and   \cref{lem:norm-characteristic-functions-Lie-group}, we get
\begin{align*}
 I_{1} &=   \dfrac{1}{|B|^{1+\beta/Q} }   \dint_{B}  \Big| b(x)-|B|^{-\alpha/Q} M_{\alpha,B}(b)(x) \Big| \mathd x      \\
 &\le   \dfrac{1}{|B|^{1+\beta/Q} }  \bigg( \dint_{B}  \Big| b(x)-|B|^{-\alpha/Q} M_{\alpha,B}(b)(x) \Big|^{s} \mathd x \bigg)^{1/s}  \|\dchi_{B}\|_{L^{s'}(\mathbb{G}) } \\
  &\le  \dfrac{C}{|B|^{\beta/Q} }  \bigg( \dfrac{1}{|B| }\dint_{B}  \Big| b(x)-|B|^{-\alpha/Q} M_{\alpha,B}(b)(x) \Big|^{s} \mathd x \bigg)^{1/s}    \\
&\le C,
\end{align*}
where the constant $C$ is independent of ball $B$.

 Now we consider $ I_{2}$. For all  $x\in B$,  it follows from \cref{lem:frac-maximal-pointwise-relation} and \cref{rem.frac-maximal-pointwise-relation}   that the pointwise estimates
\begin{align*}
   M_{\alpha} (\dchi_{B})(x)   &=  |B|^{\alpha/Q} \ \text{and}\  M_{\alpha} (b\dchi_{B})(x)   =  M_{\alpha,B} (b)(x),
\\ \intertext{and}
   M(\dchi_{B})(x)   &= \dchi_{B}(x)   =  1 \ \text{and}\  M (b\dchi_{B})(x)   =  M_{B} (b)(x).
\end{align*}
Then, for any $x\in B$,  we get
\begin{align} \label{inequ:proof-lem-non-negative-max-lip-41-3}
\begin{split}
  &\Big| |B|^{-\alpha/Q} M_{\alpha,B}(b)(x)  -M_{B}(b)(x) \Big|  \\
  &= \Big| |B|^{-\alpha/Q} M_{\alpha,B}(b)(x) -|b(x)|+ |b(x)| -M_{B}(b)(x) \Big|  \\
    &\le   |B|^{-\alpha/Q} \Big| M_{\alpha,B}(b)(x) - |B|^{\alpha/Q}|b(x)| \Big|  + \Big|  |b(x)|-M_{B}(b)(x) \Big|    \\
    &\le   |B|^{-\alpha/Q}   \big|   M_{\alpha} (b\dchi_{B})(x) -  |b(x)|  M_{\alpha} (\dchi_{B})(x) \big|  \\
    &\qquad + \big|  |b(x)| M(\dchi_{B})(x) -M (b\dchi_{B})(x)  \big|    \\
  &\le  |B|^{-\alpha/Q}  \big|[|b|,M_{\alpha}](\dchi_{B})(x) \big|  + \big| [ |b|, M] (\dchi_{B})(x)   \big|.
\end{split}
\end{align}
Since statement \labelcref{enumerate:nonlinear-frac-max-lip-4} along with \cref{lem:frac-Lie-lip-norm} gives $b\in \Lambda_{\beta}(\mathbb{G})$, which implies $|b|\in \Lambda_{\beta}(\mathbb{G})$.
Hence, we can apply  \cref{lem:frac-maximal-pointwise} to $[|b|,M_{\alpha}]$ and $[ |b|, M]$ since $|b|\in \Lambda_{\beta}(\mathbb{G})$ and $|b|\ge 0$.

By using  \cref{lem:frac-maximal-pointwise},  \cref{lem:frac-maximal-pointwise-relation} and \cref{rem.frac-maximal-pointwise-relation}, for any $x\in B$,  we have
\begin{align*}
  \big|[|b|,M_{\alpha}](\dchi_{B})(x) \big| &\le \|b\|_{\Lambda_{\beta}(\mathbb{G})} M_{\alpha+\beta} (\dchi_{B})(x)   \le C \|b\|_{\Lambda_{\beta}(\mathbb{G})} |B|^{(\alpha+\beta)/Q}
\\ \intertext{and}
 \big| [ |b|, M] (\dchi_{B})(x)   \big| &\le \|b\|_{\Lambda_{\beta}(\mathbb{G})} M_{\beta} (\dchi_{B})(x)  \le  C \|b\|_{\Lambda_{\beta}(\mathbb{G})} |B|^{\beta/Q}.
\end{align*}
Thus,   it follows from \labelcref{inequ:proof-lem-non-negative-max-lip-41-3} that
\begin{align*}
 I_{2}&=  \dfrac{1}{|B|^{1+\beta/Q} }   \dint_{B}  \Big| |B|^{-\alpha/Q} M_{\alpha,B}(b)(x)  -M_{B}(b)(x)  \Big| \mathd x   \\
    &\le  \dfrac{C}{|B|^{1+(\alpha+\beta)/Q} }   \dint_{B}  \big|[|b|,M_{\alpha}](\dchi_{B})(x) \big| \mathd x   \\
    &\qquad +\dfrac{C}{|B|^{1+\beta/Q} }   \dint_{B} \big| [ |b|, M] (\dchi_{B})(x)   \big| \mathd x   \\
    &\le C\|b\|_{\Lambda_{\beta}(\mathbb{G})}.
\end{align*}

Putting the above estimates for $I_{1}$ and $I_{2}$ into \labelcref{inequ:proof-lem-non-negative-max-lip-41-2}, we obtain \labelcref{inequ:proof-lem-non-negative-max-lip-41}.

\labelcref{enumerate:nonlinear-frac-max-lip-2} $\xLongrightarrow[]{\ \  }$ \labelcref{enumerate:nonlinear-frac-max-lip-5}:\
Assume statement \labelcref{enumerate:nonlinear-frac-max-lip-2} is true. Reasoning as in the proof of \labelcref{enumerate:nonlinear-frac-max-lip-3} $\xLongrightarrow[]{\ \  }$ \labelcref{enumerate:nonlinear-frac-max-lip-4}, we have
\begin{align} \label{inequ:proof-lem-non-negative-max-lip-25}
\begin{aligned}
 \sup_{B} \dfrac{1}{|B|^{\beta/Q}} &\left(  \dfrac{1}{|B|} \dint_{B}  |b(x) -|B|^{-\alpha/Q}M_{\alpha,B}(b)(x)  |^{q} \mathd x \right)^{1/q}  <\infty
 \end{aligned}
\end{align}
for any $q$ for which there exists a $p$ such that $\frac{1}{q} =\frac{1}{p} - \frac{\alpha+\beta}{Q}$.

For any $s \in   [1,\infty)$, choosing an $r>Q/(Q-\beta)>1$, we have $1<rs(Q-\beta)/Q< rs$. Set $q=rs$ and define $p$ by $\frac{1}{p} =\frac{1}{q} + \frac{\alpha+\beta}{Q}$.
Noting that
\begin{align*}
    \dfrac{1}{s} &=    \dfrac{1}{rs} + \dfrac{1}{r's} = \dfrac{1}{q} + \dfrac{1}{r's},
\end{align*}
it follows from  \cref{lem:holder-inequality-Lie-group}, \labelcref{inequ:proof-lem-non-negative-max-lip-25} and   \cref{lem:norm-characteristic-functions-Lie-group} that
\begin{align*}
    \dfrac{1}{|B|^{\beta/Q}} & \left(  \dfrac{1}{|B|} \dint_{B} \big |b(x) -|B|^{-\alpha/Q}M_{\alpha,B}(b)(x)  \big |^{s} \mathd x \right)^{1/s}  \\
 &\le \dfrac{C}{|B|^{1/s+\beta/Q}}  \left( \dint_{B}  \big |b(x) -|B|^{-\alpha/Q}M_{\alpha,B}(b)(x)  \big |^{q} \mathd x \right)^{\frac{1}{q}}  \|\dchi_{B}\|_{L^{r's}(\mathbb{G}) }  \\
 &\le \dfrac{C}{|B|^{1/s-1/q-\frac{1}{r's}}}   \dfrac{1}{|B|^{\beta/Q}} \left( \dfrac{1}{|B|} \dint_{B}  \big |b(x) -|B|^{-\alpha/Q}M_{\alpha,B}(b)(x)  \big |^{q} \mathd x \right)^{\frac{1}{q}}  \\
 &\le C,
\end{align*}
which is what we want.

The proof  is completed.
\end{proof}

\subsection{Proof of \cref{thm:frac-max-lip}}


\begin{proof}[Proof of \cref{thm:frac-max-lip}]
Since the implications \labelcref{enumerate:thm-frac-max-lip-2} $\xLongrightarrow[]{\ \  }$ \labelcref{enumerate:thm-frac-max-lip-3}  and \labelcref{enumerate:thm-frac-max-lip-5} $\xLongrightarrow[]{\ \  }$ \labelcref{enumerate:thm-frac-max-lip-4}    follows readily,
we only need to prove \labelcref{enumerate:thm-frac-max-lip-1} $\xLongrightarrow[]{\ \  }$ \labelcref{enumerate:thm-frac-max-lip-2}, \labelcref{enumerate:thm-frac-max-lip-3} $\xLongrightarrow[]{\ \  }$ \labelcref{enumerate:thm-frac-max-lip-4},  \labelcref{enumerate:thm-frac-max-lip-4} $\xLongrightarrow[]{\ \  }$ \labelcref{enumerate:thm-frac-max-lip-1}
and \labelcref{enumerate:thm-frac-max-lip-2} $\xLongrightarrow[]{\ \  }$ \labelcref{enumerate:thm-frac-max-lip-5} (the proof structure is also shown in \Cref{fig:ps-equivalent-non}).

\labelcref{enumerate:thm-frac-max-lip-1} $\xLongrightarrow[]{\ \  }$ \labelcref{enumerate:thm-frac-max-lip-2}:\
Let    $b\in \Lambda_{\beta}(\mathbb{G})$, then, using \cref{def.lip-space} \labelcref{enumerate:def-lip-1}, we have
\begin{align}  \label{inequ:proof-frac-main-1-1}
\begin{aligned}
M_{\alpha,b} (f)(x) &= \sup_{B\ni x}   \dfrac{1}{|B|^{1-\alpha/Q}}  \dint_{B} |b(x)-b(y)| |f(y)| \mathd y  \\
&\le C\|b\|_{\Lambda_{\beta}(\mathbb{G})} \sup_{B\ni x}  \dfrac{1}{|B|^{1-\alpha/Q}}    \dint_{B}|\rho(y^{-1}x)|^{\beta} |f(y)| \mathd y     \\
&\le C\|b\|_{\Lambda_{\beta}(\mathbb{G})} \sup_{B\ni x}  \dfrac{1}{|B|^{1-(\alpha+\beta)/Q}} \dint_{B}  |f(y)| \mathd y        \\
&\le C\|b\|_{\Lambda_{\beta}(\mathbb{G})}  M_{\alpha+\beta} (f)(x).
\end{aligned}
\end{align}
Therefore, assertion \labelcref{enumerate:thm-frac-max-lip-2} follows from  \cref{lem:frac-maximal-kokilashvili1989fractional} and  \labelcref{inequ:proof-frac-main-1-1}.

\labelcref{enumerate:thm-frac-max-lip-3} $\xLongrightarrow[]{\ \ \ \ }$ \labelcref{enumerate:thm-frac-max-lip-4}:\
 For any fixed ball $B\subset \mathbb{G}$,  we have
\begin{align*}
  |b(x)-b_{B}|  &\le   \dfrac{1}{ |B| }    \dint_{B}   |b(x)-b(y)|  \mathd y  \\
   &= \dfrac{1}{ |B| }    \dint_{B}   |b(x)-b(y)| \dchi_{B} (y) \mathd y   \\
   &\le \dfrac{1}{ |B|^{\alpha/Q} }  M_{\alpha,b}(\dchi_{B})(x)
\end{align*}
 for all $x\in B$.
Then, for all $x\in \mathbb{G}$,  we get
\begin{align*}
|(b(x)-b_{B})\dchi_{B} (x)|  \le  \dfrac{1}{ |B|^{\alpha/Q} } M_{\alpha,b}(\dchi_{B})(x).
\end{align*}

  Since $ M_{\alpha,b} $ is bounded from    $L^{p}(\mathbb{G})$ to $L^{q}(\mathbb{G})$. Then,   by using assertion \labelcref{enumerate:thm-frac-max-lip-3} and \cref{lem:norm-characteristic-functions-Lie-group},  for any ball $B\subset \mathbb{G}$,  one obtains
\begin{align*}
 \dfrac{1}{|B|^{\beta/Q}} \Big( \dfrac{1}{|B|} \dint_{B} |b(x)-b_{B}|^{q} \mathd x \Big)^{1/q}
&\le \dfrac{1}{ |B|^{(\alpha+\beta)/Q}} \Big( \dfrac{1}{|B|} \dint_{B} \big( M_{\alpha,b}(\dchi_{B})(x) \big)^{q}  \mathd x \Big)^{1/q}  \\
&\le  \dfrac{C}{ |B|^{1/q+(\alpha+\beta)/Q}}   \|\dchi_{B}\|_{L^{p}(\mathbb{G})}   \\
&\le C.
\end{align*}
which gives  \labelcref{inequ:frac-max-lip-4} for $s  = q $ since $B$ is arbitrary and $C$ is independent of $B$.

\labelcref{enumerate:thm-frac-max-lip-4} $\xLongrightarrow[]{\ \ \ \ }$ \labelcref{enumerate:thm-frac-max-lip-1}:\
For any ball $B \subset \mathbb{G}$, it follows from H\"{o}lder's inequality (see \cref{lem:holder-inequality-Lie-group}), \cref{lem:norm-characteristic-functions-Lie-group} and assertion \labelcref{enumerate:thm-frac-max-lip-4} that
\begin{align*}
 \dfrac{1}{ |B|^{1+\beta/Q}}  \dint_{B} |b(x)- b_{B}| \mathd x
&\le  \dfrac{C}{|B|^{1+\beta/Q}} \Big( \dint_{B} |b(x)-b_{B}|^{q} \mathd x \Big)^{1/q} \Big(\dint_{B}   \dchi_{B}(x)    \mathd x  \Big)^{1/q'}   \\
&\le  \dfrac{C}{|B|^{\beta/Q}} \Big( \dfrac{1}{|B|} \dint_{B} |b(x)-b_{B}|^{q} \mathd x \Big)^{1/q} \\
&\le C.
\end{align*}
It follows from \cref{lem:2.2-li2003lipschitz}  and \cref{def.lip-space}  that $b\in \Lambda_{\beta}(\mathbb{G})$ since $B$ is an arbitrary ball in $\mathbb{G}$.

 \labelcref{enumerate:thm-frac-max-lip-2} $\xLongrightarrow[]{\ \  }$ \labelcref{enumerate:thm-frac-max-lip-5}:\ Similar to the course of the proof of \labelcref{enumerate:nonlinear-frac-max-lip-2} $\xLongrightarrow[]{\ \  }$ \labelcref{enumerate:nonlinear-frac-max-lip-5}, thus, we omit it.

The proof of \cref{thm:frac-max-lip} is completed.
\end{proof}



 \subsubsection*{Funding information:}

 This work was  financially supported by the Scientific Research Fund of APU (No.S022022177, for Zhao), Project of Heilongjiang Province Science and Technology Program (No.2019-KYYWF-0909, for Wu), the National Natural Science Foundation of China (No.11571160, for Wu), the Reform and Development Foundation for Local Colleges and Universities of the Central Government(No.2020YQ07, for Wu) and the Scientific Research Fund of Mudanjiang Normal University (No.D211220637, for Wu).

%
 \subsubsection*{Data availability statement:}  
 This manuscript has no associate data.


%
%
%


\phantomsection
\addcontentsline{toc}{section}{References}
%

\bibliography{wu-reference}

\end{document}